\documentclass[12pt]{amsart}
\usepackage{ amsthm, amscd, amsfonts, amssymb, graphicx, color}
\usepackage[bookmarksnumbered, colorlinks, plainpages]{hyperref}
\hypersetup
{colorlinks=true,linkcolor=black,pdfborder=black,citecolor=black,
   naturalnames=false,citecolor=red,pdfview=FitH,pdfstartview=FitH
 }

\newtheorem{theorem}{{Theorem}}[]

\newtheorem{proposition}[theorem]{{Proposition}}

\theoremstyle{definition}
\newtheorem{definition}[theorem]{{Definition\ }}

\begin{document}

\title{On the Feichtinger Conjecture}
\author{ P. G\u avru\c ta }
\address{ "Politehnica" University of Timi\c soara, Department of Mathematics,
 300006, Timi\c soara, Romania.}
\email{pgavruta@yahoo.com}



\abstract{We prove the Feichtinger Conjecture for a class of Bessel sequences of unit norm vectors in a Hilbert space. Also, we prove that every Bessel sequence of unit vectors in a Hilbert space can be partitioned into finitely many uniformly separated sequences.} }

\maketitle

\section{Introduction}
There are many variations of the Feichtinger Conjecture, all equivalent with the following:\\

\textit{Every Bessel sequence of unit vectors in a Hilbert space can be partitioned into finitely many Riesz sequences.}\\

For details on the Feichtinger Conjecture and the connection with other problems, see \cite{PGCasazza1}, \cite{PGCasazza2}, \cite{Chalendar}, \cite{Lata}, \cite{Lata1} and references of these papers.\\

We denote by $\mathcal{H}$ a Hilbert space and $\mathcal{F}=\{f_n\}_{n\in\mathbb{N}}\subset\mathcal{H}.$ We say that $\mathcal{F}$ is a \textit{Bessel sequence} if there exists $B>0$ so that $$\sum_{n=0}^{\infty}|\langle x,f_n\rangle|^2\leq B\|x\|^2, \quad\forall x\in\mathcal{H}.$$
$B$ is called Bessel constant for $\mathcal{F}$.\\

We say that $\mathcal{F}$ is a \textit{frame} for $\mathcal{H}$ if it is a Bessel sequence and there exists $A>0$ so that $$A\|x\|^2\leq\displaystyle \sum_{n=0}^{\infty} |\langle x,f_n\rangle|^2,\quad\forall x\in\mathcal{H}.$$\\

For important applications of frames see the references of the paper \cite{Gavruta}.\\

We say that $\mathcal{F}$ is a \textit{Riesz sequence (or Riesz basic sequence)} if there are $A,B>0$ such that $$A\sum |c_k|^2\leq\|\sum c_kf_k\|\leq B\sum|c_k|^2,$$
for any finite sequence $(c_k).$ Riesz sequences are particular cases of frames (see \cite{Christensen}).\\

Let be $I\subset\mathbb{N}.$ If $\mathcal{F}$ is a Bessel sequences in $\mathcal{H},$ then $\mathcal{F}_I=\{f_n\}_{n\in I}$ is clearly also a Bessel sequence in $\mathcal{H}$.\\

In \cite{Christensen}, O. Christensen, using Schur's test, give conditions on a sequence $\{f_n\}_{n=0}^\infty$ to be a Bessel sequence, that it only involves inner products between the elements $\{f_n\}_{n=0}^\infty$:
\begin{proposition}\cite{Christensen} Let $\{f_n\}_{n=0}^\infty$ be a sequence in $\mathcal{H}$ and assume that there exists a constant $B>0$ such that $$\sum_{k=0}^\infty|\langle f_j,f_k\rangle|\leq B,\quad\forall j\in\mathbb{N}.$$
Then $\{f_n\}_{n=0}^\infty$ is a Bessel sequence with bound $B$.
\end{proposition}
We call this sequences \textit{Bessel-Schur sequences}.\\

The intrinsically localized sequences, introduced by K. Gr\"ochenig in \cite{Grochenig}, are particular cases of Bessel-Schur sequences. In the same paper, he proves that every localized frame is a finite union of Riesz sequences.

On the other hand, we recall the following definition:
\begin{definition}\cite{Chalendar}
A sequence $\{f_n\}_{n\in I}$ of unit vectors in $\mathcal{H}$ is called \textit{separated} if there exists a constant $\gamma<1$ such that $$|\langle f_n,f_k\rangle|\leq \gamma$$
for any $n, k\in\mathbb{N}$, $n\neq k$.
\end{definition}
In \cite{Chalendar}, the authors, among others, give the following result:
\begin{theorem}
Let $\mathcal{H}$ be a Hilbert space and let $\{f_n\}_{n\in I}$ be a Bessel sequence of unit vectors in $\mathcal{H}.$ Then $\{f_n\}_{n\in I}$ can be partitioned into finitely many separated Bessel sequences.
\end{theorem}
In the following, we prove that the Bessel-Schur sequences satisfies the Feichtinger Conjecture.
Also, we prove that every Bessel sequence of unit vectors in a Hilbert space can be partitioned into finitely many uniformly separated sequences.
\section{The results}
First, we give a condition for a Bessel sequence of unit vectors to be a Riesz sequence.
\begin{theorem}\label{Theorem1} Let $\mathcal{F}_I=\{f_n\}_{n\in I}$ be a Bessel sequence of unit vectors. We suppose that $$\sigma:=\sup_{j\in I}\sum_{\substack{
i\in I \\
i\neq j
}}|\langle f_i,f_j\rangle|<1.$$
Then $\mathcal{F}_I$ is a Riesz sequence.
\end{theorem}

\begin{proof}If $\mathcal{F}_I$ is a Bessel sequence in $\mathcal{H}$, then the following operators are linear and bounded:
$$T:l^2(I)\rightarrow\mathcal{H},\quad T(c_i)=\sum_{i\in I} c_if_i\quad \mbox{(synthesis operator)}$$
$$\Theta:\mathcal{H}\rightarrow l^2(I),\quad \Theta x=\{\langle x, f_i\rangle\}_{i\in I}\quad \mbox{(analysis operator)}$$
Moreover, $\Theta$ is the adjoint of $T$ (see \cite{Christensen}).\\

For $c=(c_n)_{n\in I}\in l^2(I)$, we have
\begin{align*}
(\Theta T)(c)&=\bigg\{\langle \sum_{
k\in I
}c_kf_k,f_j\rangle\bigg\}_{j\in I}\\
             &=\bigg\{\sum_{
k\in I
} c_k \langle f_k,f_j\rangle\bigg\}_{j\in I}\\
 \end{align*}
hence
 $$(\Theta T)(c)-c=\bigg\{ \sum_{k\neq j} c_k \langle f_k,f_j\rangle\bigg\}_{j\in I}.$$\\
 It follows, with Cauchy-Schwartz inequality,
 \begin{align*}
\|(\Theta T)(c)-c\|_2^2&=\sum_{j\in I}\bigg|\sum_{k\neq j}c_k\langle f_k,f_j\rangle\bigg|^2\\
                       &\leq\sum_{j\in I}\bigg(\sum_{\substack{k\in I \\k\neq j}}|c_k||\langle
                           f_k,f_j\rangle|^{1/2}\cdot|\langle f_k,f_j\rangle|^{1/2}\bigg)^2\\
                       & \leq\sum_{j\in I}\bigg(\sum_{\substack{k\in I \\k\neq j}}|c_k|^2|\langle
                            f_k,f_j\rangle|\bigg)\bigg(\sum_{\substack{k\in I \\k\neq j}}|\langle f_k,f_j\rangle| \bigg)\\
                        &\leq\sigma\sum_{j\in I}\bigg(\sum_{\substack{k\in I \\k\neq j}}|c_k|^2|\langle f_k,f_j\rangle|\bigg)\\.
   \end{align*}
   Changing the order of summation, we obtain
 \begin{align*}
\|(\Theta T)(c)-c\|_2^2&\leq\sigma\sum_{k\in I}|c_k|^2\sum_{\substack{j\in I \\j\neq k}}\langle f_k,f_j\rangle|\\
                       &\leq\sigma^2\|c\|_2^2,
\end{align*}
so $\|\Theta T-I\|\leq \sigma<1.$\\

 Therefore $\Theta T$ is invertible, thus $\Theta$ is surjective. It follows that $\mathcal{F}_I$ is a Riesz-Fischer sequence. From Theorem 3 in \cite{Young}, Ch. 4, Sec. 2, we have that there exists $A>0$ so that $$A\sum|c_k|^2\leq\|\sum c_kf_k\|^2$$ for every finite sequence $(c_k).$\\

Since $\mathcal{F}_I$ is a Bessel sequence, we have
 $$\|\sum c_kf_k\|^2\leq B\sum|c_k|^2$$ for $(c_k)$ finite sequence (see \cite{Christensen}). So $\mathcal{F}_I$ is a Riesz sequence.
\end{proof}

\begin{theorem}Every Bessel-Schur sequence of unit vectors is union of finite Riesz sequences.
\end{theorem}
\begin{proof}
Let $j\in\mathbb{N}$ fixed. We have:
$$\sum_{i=0}^{\infty}|\langle f_j,f_i\rangle|\leq B,$$
hence \begin{equation}\label{eq1}\sum_{\substack{
i=0 \\
i\neq j
}}|\langle f_j,f_i\rangle|\leq B-1,\quad\mbox{for any $j\in\mathbb{N}$}.
\end{equation}
We denote $$a_{ij}=\begin{cases}
|\langle f_j,f_i\rangle| &,\quad j\neq i\\
0 &,\quad j=i
\end{cases}
$$
We have $a_{ij}=a_{ji}\geq 0$ and $a_{ii}=0.$\\

The relation (\ref{eq1}) is equivalent with $$\sup_{j\in\mathbb{N}}\sum_{i\in\mathbb{N}}a_{ij}\leq B-1.$$
By Mills' Lemma (see \cite{Garnett},Ch.X or \cite{Thomas}) there is a partition $\mathbb{N}=I_1\cup I_2$ such that $$\sup_{j\in I_p}\sum_{i\in I_p}a_{ij}\leq\frac{B-1}{2};\quad p=1,2.$$
By iteration, for any $m\geq 1$, there is a partition $\mathbb{N}=I_1\cup I_2\cup\ldots\cup I_{2^m}$ such that $$\sup_{j\in I_p}\sum_{
i\in I_p}a_{ij}\leq\frac{B-1}{2^m},\quad \forall p=1,2,\ldots 2^m.$$
We take $m$ so that $\frac{B-1}{2^m}<1$ and apply Theorem \ref{Theorem1}.
\end{proof}
\section{An equivalent form of the Feichtinger conjecture}
We consider the following class of sequences.
\begin{definition}\label{Definition1} Let $\mathcal{F}_I=\{f_n\}_{n\in I}$ be a sequence of unit vectors. We say that this sequence is \textit{uniformly separated} if the following condition holds $$\eta:=\sup_{j\in I}\sum_{\substack{
i\in I \\
i\neq j
}}|\langle f_i,f_j\rangle|^2<1.$$
\end{definition}

The following result is a refinement of a result from \cite{Chalendar}.
\begin{theorem}Every Bessel sequence of unit vectors is union of finite uniformly separated sequences.
\end{theorem}
\begin{proof}Let $\mathcal{F}$ be a Bessel sequence of unit vectors:
$$\sum_{i=0}^{\infty}|\langle x,f_i\rangle|^2\leq B\|x\|^2,\quad\forall x\in\mathcal{H}.$$
Let $j\in\mathbb{N}$ fixed. We take $x=f_j:$
$$\sum_{i=0}^{\infty}|\langle f_j,f_i\rangle|^2\leq B\|f_j\|^2=B,$$
hence \begin{equation}\label{eq2}\sum_{\substack{
i=0 \\
i\neq j
}}|\langle f_j,f_i\rangle|^2\leq B-1,\quad\mbox{for any $j\in\mathbb{N}$}.
\end{equation}
It is clear that $B\geq 1.$\\
We denote $$a_{ij}=\begin{cases}
|\langle f_j,f_i\rangle|^2 &,\quad j\neq i\\
0 &,\quad j=i
\end{cases}
$$
We have $a_{ij}=a_{ji}\geq 0$ and $a_{ii}=0.$\\

The relation (\ref{eq2}) is equivalent with $$\sup_{j\in\mathbb{N}}\sum_{i\in\mathbb{N}}a_{ij}\leq B-1.$$
By Mills' Lemma (see \cite{Garnett},Ch.X or \cite{Thomas}) there is a partition $\mathbb{N}=I_1\cup I_2$ such that $$\sup_{j\in I_p}\sum_{i\in I_p}a_{ij}\leq\frac{B-1}{2};\quad p=1,2.$$
By iteration, for any $m\geq 1$, there is a partition $\mathbb{N}=I_1\cup I_2\cup\ldots\cup I_{2^m}$ such that $$\sup_{j\in I_p}\sum_{
i\in I_p}a_{ij}\leq\frac{B-1}{2^m},\quad \forall p=1,2,\ldots 2^m.$$
We take $m$ so that $\frac{B-1}{2^m}<1$ and apply Definition \ref{Definition1}.
\end{proof}
From the above Theorem, we obtain the following equivalent form of the Feichtinger Conjecture:\\

\textit{Every uniformly separated Bessel sequence of unit norm vectors can be partitioned into finitely many Riesz sequences.}

\end{document}